\documentclass[a4]{amsart}

\usepackage{amssymb}
\usepackage[all]{xy}
\usepackage{amsmath, amsfonts, amsthm , amssymb}

\input xypic
\usepackage{tikz}

\newtheorem{theorem}{Theorem}[section]
\newtheorem{proposition}[theorem]{Proposition}
\newtheorem{lemma}[theorem]{Lemma}
\newtheorem{corollary}[theorem]{Corollary}

\theoremstyle{definition}

\theoremstyle{remark}
\newtheorem{remark}[theorem]{Remark}

\newcommand{\uxa}{\ensuremath{(\underline{X},\underline{A})}}
\newcommand{\ux}{\ensuremath{(\underline{X},\underline{\ast})}}
\newcommand{\cyy}{\ensuremath{(\underline{CY},\underline{Y})}}
\newcommand{\cyyk}{\ensuremath{(\underline{CY},\underline{Y})^{K}}}

\newcommand{\clxx}{\ensuremath{(\underline{C\Omega X},\underline{\Omega X})}}

\newcommand{\starr}{\mathop{\rm star}\nolimits}
\newcommand{\link}{\mathop{\rm link}\nolimits}

\newcommand{\cpinf}{\ensuremath{\mathbb{C}P^{\infty}}}

\newcommand{\zk}{\ensuremath{\mathcal{Z}_{K}}}

\newcommand{\cohlgy}[1]{\ensuremath{H^{*}(#1)}}

\newcounter{bean}
\newenvironment{letterlist}{\begin{list}{\rm ({\alph{bean}})}
      {\usecounter{bean}\setlength{\rightmargin}{\leftmargin}}}
      {\end{list}}

\newcommand{\namedright}[3]{\ensuremath{#1\stackrel{#2}
 {\longrightarrow}#3}}
\newcommand{\nameddright}[5]{\ensuremath{#1\stackrel{#2}
 {\longrightarrow}#3\stackrel{#4}{\longrightarrow}#5}}
\newcommand{\namedddright}[7]{\ensuremath{#1\stackrel{#2}
 {\longrightarrow}#3\stackrel{#4}{\longrightarrow}#5
  \stackrel{#6}{\longrightarrow}#7}}

\newcommand{\larrow}{\relbar\!\!\relbar\!\!\rightarrow}
\newcommand{\llarrow}{\relbar\!\!\relbar\!\!\larrow}

\newcommand{\llnamedright}[3]{\ensuremath{#1\stackrel{#2}
 {\llarrow}#3}}
\newcommand{\llnameddright}[5]{\ensuremath{#1\stackrel{#2}
 {\llarrow}#3\stackrel{#4}{\llarrow}#5}}

\newcommand{\qqed}{\hfill\Box}

\def\MF{\mathrm{MF}}

\begin{document}
\title[Polyhedral products associated with flag complexes]
    {The homotopy theory of polyhedral products associated with flag complexes}
\author[Taras Panov]{Taras Panov}
\address{Department of Mathematics and Mechanics, Moscow
       State University, Leninskie Gory, 119991 Moscow, Russia,
       \newline\indent Institute for Theoretical and Experimental Physics,
       Moscow, Russia \quad \emph{and}
       \newline\indent Institute for Information Transmission Problems,
       Russian Academy of Sciences}
\email{tpanov@mech.math.msu.su}
\author{Stephen Theriault}
\address{Mathematical Sciences, University of Southampton,
       Southampton SO17 1BJ, United Kingdom}
\email{S.D.Theriault@soton.ac.uk}

\date{}

\begin{abstract}
If $K$ is a simplicial complex on $m$ vertices the flagification
of $K$ is the minimal flag complex $K^{f}$ on the same vertex set
that contains $K$. Letting $L$ be the set of vertices, there is a
sequence of simplicial inclusions
\(\nameddright{L}{}{K}{}{K^{f}}\). This induces a sequence of maps
of polyhedral products
\(\nameddright{\uxa^{L}}{g}{\uxa^{K}}{f}{\uxa^{K^{f}}}\). We show
that $\Omega f$ and $\Omega f\circ\Omega g$ have right homotopy
inverses and draw consequences. For a flag complex $K$ the
polyhedral product of the form $\cyy^{K}$ is a co-$H$-space if and
only if the $1$-skeleton of $K$ is a chordal graph, and we deduce
that the maps $f$ and $f\circ g$ have right homotopy inverses in
this case.
\end{abstract}

\maketitle

\section{Introduction}
\label{sec:intro}

The purpose of this paper is to investigate the homotopy theory of
polyhedral products associated with flag complexes.
Polyhedral products have received considerable attention recently as they
unify diverse constructions from several seemingly separate areas of
mathematics: toric topology (moment-angle complexes),
combinatorics (complements of complex coordinate subspace arrangements),
commutative algebra (the Golod property of monomial rings), complex
geometry (intersections of quadrics), and geometric group theory
(Bestvina-Brady groups).

To be precise, let $K$ be a simplicial complex on the vertex set
$[m]=\{1,2,\ldots,m\}$. For $1\leq i\leq m$, let $(X_{i},A_{i})$ be a
pair of pointed $CW$-complexes, where $A_{i}$ is a pointed $CW$-subcomplex
of $X_{i}$. Let $\uxa=\{(X_{i},A_{i})\}_{i=1}^{m}$ be the sequence of pairs.
For each simplex $\sigma\in K$, let $\uxa^{\sigma}$ be the
subspace of $\prod_{i=1}^{m} X_{i}$ defined by
\[\uxa^{\sigma}=\prod_{i=1}^{m} Y_{i}\qquad
       \mbox{where}\qquad Y_{i}=\left\{\begin{array}{ll}
                                             X_{i} & \mbox{if $i\in\sigma$} \\
                                             A_{i} & \mbox{if $i\notin\sigma$}.
                                       \end{array}\right.\]
The \emph{polyhedral product} determined by \uxa\ and $K$ is
\[\uxa^{K}=\bigcup_{\sigma\in K}\uxa^{\sigma}\subseteq\prod_{i=1}^{m} X_{i}.\]
For example, suppose each $A_{i}$ is a point. If $K$ is a disjoint
union of $m$ points then $(\underline{X},\underline{\ast})^{K}$ is
the wedge $X_{1}\vee\cdots\vee X_{m}$, and if $K$ is the standard
$(m-1)$-simplex then $(\underline{X},\underline{\ast})^{K}$ is the
product $X_{1}\times\cdots\times X_{m}$.

The combinatorics of $K$ informs greatly on the homotopy theory of
$\uxa^{K}$. One notable family of simplicial complexes is the collection
of flag complexes. A simplicial complex $K$ is flag if any set of
vertices of $K$ which are pairwise connected by edges spans a
simplex. Flag complexes are important in graph theory, where they are referred
to as clique complexes, in the study of metric spaces, where they are referred
to as Rips complexes, and in geometric group theory, where they are referred
to as Gromov's no-$\triangle$ complexes.

The \emph{flagification} of $K$, denoted~$K^f$, is the
minimal flag complex on the same set $[m]$ that contains~$K$. We
therefore have a simplicial inclusion $K\to K^f$. For example, the
$(m-1)$-simplex $\varDelta^{m-1}$, consisting of all subsets
of~$[m]$, is flag, while its boundary $\partial\varDelta^{m-1}$,
consisting of all proper subsets of~$[m]$, is flag only for $m=2$.
The flagification of $\partial\varDelta^{m-1}$ with $m>2$
is~$\varDelta^{m-1}$. An \emph{$m$-cycle} (the boundary of an
$m$-gon) is flag whenever $m>3$.

The main result of the paper is the following.

\begin{theorem}
   \label{main}
   Let $K$ be a simplicial complex on the vertex set $[m]$, let
   $K^{f}$ be the flagification of~$K$, and let $L$ be the simplicial
   complex given by $m$ disjoint points. Let $\uxa=\{(X_{i},A_{i})\}_{i=1}^{m}$ 
   be a sequence of pairs of pointed $CW$-complexes, where $A_{i}$ is a 
   pointed $CW$-subcomplex of $X_{i}$. Let
   \(\nameddright{\uxa^{L}}{g}{\uxa^{K}}{f}{\uxa^{K^{f}}}\)
   be the maps of polyhedral products induced by the maps of simplicial complexes
   \(\nameddright{L}{}{K}{}{K^{f}}\).
   Then the following hold:
   \begin{letterlist}
      \item the map $\Omega f$ has a right homotopy inverse;
      \item the composite $\Omega f\circ\Omega g$ has a right homotopy inverse.
   \end{letterlist}
\end{theorem}

In particular, consider the special case when each $A_{i}$ is a point. Write
$\ux$ for $\uxa$ and notice that $\ux^{L}=X_{1}\vee\cdots\vee X_{m}$.
If $K$ is a flag complex on the vertex set $[m]$ then the simplicial map
\(\namedright{L}{}{K}\)
induces a map
\(f\colon\namedright{X_{1}\vee\cdots\vee X_{m}=\ux^{L}}{}{\ux^{K}}\).
By Theorem~\ref{main}, $\Omega f$ has a right homotopy inverse. That is,
$\Omega\ux^{K}$ is a retract of $\Omega(X_{1}\vee\cdots\vee X_{m})$.
This informs greatly on the homotopy theory of $\Omega\ux^{K}$ since
the homotopy type of $\Omega(X_{1}\vee\cdots\vee X_{m})$ has been well
studied; in particular, in the special case when each $X_{i}$ is a suspension
the Hilton-Milnor Theorem gives an explicit homotopy decomposition of
the loops on the wedge. Theorem~\ref{main} also greatly
generalizes~\cite[Theorem 5.3]{GPTW}, which stated that
such a retraction exists in the special case when each $X_{i}=\cpinf$
provided spaces and maps have been localized at a prime $p\neq 2$.

Theorem~\ref{main} can be improved in certain cases. In Section~\ref{sec:coH}
we consider polyhedral products of the form $\cyy^{K}$, where $CY$ is the
cone on $Y$, and identify the class of flag complexes $K$ for which $\cyy^{K}$ is a
co-$H$-space. As a corollary, we obtain conditions that allow for a
delooping of the statement of Theorem~\ref{main}. In Section~\ref{sec:whitehead}
we relate Theorem~\ref{main} to Whitehead products. First, we consider polyhedral
products of the form $\ux^{K}$ with flag $K$ whose $1$-skeleton is a chordal graph,
and obtain a generalisation of Porter's description of the homotopy fiber of the
inclusion of an $m$-fold wedge into a product in terms of Whitehead brackets.
Second, we consider the loop space $\Omega(\underline{S},\underline{\ast})^{K}$
on a polyhedral product formed from spheres for an arbitrary flag complex~$K$,
and obtain a generalisation of the Hilton--Milnor Theorem.

The research of the first author was carried out at the
Institute for Information Transmission Problems of Russian Academy
of Sciences and was supported by the Russian Science Foundation
(grant no. 14-50-00150). The second author would like to thank the
Royal Society for the award of an International Exchanges Grant
which helped make this research possible. The authors would also 
like to thank the referee for making several helpful comments.

\section{Combinatorial preparation}
\label{sec:combinatorics}

This section records the combinatorial information that will be
needed. We begin with some definitions.
Let $K$ be an abstract \emph{simplicial complex} on the set
$[m]=\{1,2,\ldots,m\}$, i.\,e. $K$ is a collection of subsets
$\sigma\subseteq [m]$ such that for any $\sigma\in K$ all subsets of
$\sigma$ also belong to~$K$. We refer to $\sigma\in K$ as a
\emph{simplex} (or a \emph{face}) of~$K$ and denote by $|\sigma|$
the number of elements in~$\sigma$. We always assume that the
empty set $\varnothing$ belongs to~$K$. We do not assume that $K$
contains all one-element subsets $\{i\}\subseteq[m]$. We refer to
$\{i\}\in K$ as a \emph{vertex} of~$K$, and refer to $\{i\}\notin
K$ as a \emph{ghost vertex}. We say that $K$ is a simplicial
complex \emph{on the vertex set}~$[m]$ when there are no ghost
vertices.

Let $K$ be a simplicial complex on the set~$[m]$.
For a vertex $v\in K$, the \emph{star}, \emph{restriction} (or
\emph{deletion}) and \emph{link} of $v$ are the subcomplexes
\[\begin{array}{rcl}
     \starr_{K}(v) & = & \{\tau\in K\mid \{v\}\cup\tau\in K\}; \\
     K\setminus v & =
          & \{\tau\in K\mid \{v\}\cap\tau=\varnothing\}; \\
     \link_{K}(v) & = & \starr_{K}(v)\cap K\backslash v.
   \end{array}
\]
Throughout the paper we follow the convention of regarding
$\starr_{K}(v)$ as a simplicial complex on the same set $[m]$
as~$K$, while regarding $K\setminus v$ and $\link_K v$ as
simplicial complexes on the set $[m]\setminus v$. This implies
that $\starr_{K}(v)$ and $\link_K v$ may have ghost vertices even
if $K$ does not.

The \emph{join} of two simplicial complexes $K_{1},K_{2}$ on
disjoint sets is the simplicial complex
\[K_{1}\ast K_{2}=\{\sigma_{1}\cup \sigma_{2}\mid \sigma_{i}\in K_{i}\}.\]
From the definitions, it follows that $\starr_{K}(v)$ is a join,
\[\starr_{K}(v)=\{v\}\ast\link_{K}(v),\]
and there is a pushout
\begin{equation}\label{Kpoagain}
\diagram
         \link_{K}(v)\rto\dto & \starr_{K}(v)\dto \\
         K\backslash v\rto & K.
\enddiagram
\end{equation}

A \emph{non-face} of $K$ is a subset $\omega\subseteq [m]$ such
that $\omega\notin K$. A \emph{missing face} (a \emph{minimal
non-face}) of $K$ is an inclusion-minimal non-face of $K$, that
is, a subset $\omega\subseteq[m]$ such that $\omega$ is not a
simplex of~$K$, but every proper subset of $\omega$ is a simplex
of~$K$. A ghost vertex is therefore a missing face consisting of
one element. Denote the set of missing faces of $K$ by $\MF(K)$. For a
subset $\omega\subseteq[m]$, let $\partial\omega$ denote the
collection of proper subsets of~$\omega$. Observe that
$\omega\in\MF(K)$ if and only if $\omega\notin K$ but
$\partial\omega\subseteq K$.

A simplicial complex $K$ on the set $[m]$ is called a
\emph{flag complex} if each of its missing faces consists of at
most two elements. Equivalently, $K$ is flag if any set of
vertices of $K$ which are pairwise connected by edges spans a
simplex. Every flag complex $K$ is determined by its $1$-skeleton
$K^1$, and is obtained from the graph $K^1$ by filling in all
cliques (complete subgraphs) by simplices.

\begin{lemma}
   \label{flaglink}
   Let $K$ be a flag complex on the set $[m]$ and let $v$ be
   a vertex  of $K$. If $\omega\in\MF(\link_{K}(v))$ and $|\omega|\ge2$, then
   $\omega\in\MF(K\backslash\{v\})$.
\end{lemma}

\begin{proof}
Suppose not. Then there is a missing face $\omega$ of
$\link_{K}(v)$ with $\omega\in K\backslash\{v\}$ and $|\omega|\ge2$. Therefore,
$\partial\omega\subseteq\link_{K}(v)$ but
$\omega\notin\link_{K}(v)$. Since $\omega\in K\backslash\{v\}$, we
also have $\omega\in K$. On the other hand, as
$\starr_{K}(v)=\link_{K}(v)\ast\{v\}$, we have
$\partial\omega\ast\{v\}\subseteq\starr_{K}(v)$, and so
$\partial\omega\ast\{v\}\subseteq K$. Therefore
$\partial\omega\ast\{v\}\cup\omega\subseteq K$.

Observe that $\partial\omega\ast\{v\}\cup\omega=\partial\tau$
where~$\tau=\omega\ast\{v\}$. Thus $\partial\tau\subseteq K$. As
$K$ is flag and $|\omega\ast\{v\}|>2$, this implies that
$\tau=\omega\ast\{v\}\in K$. Hence, $\omega\in\link_K v$, a
contradiction.
\end{proof}

\begin{lemma}
   \label{slrflag}
   Let $K$ be a flag complex on the set $[m]$ and let $v$ be a vertex
   of $K$. Then $K\backslash\{v\}$, $\starr_{K}(v)$ and $\link_{K}(v)$ are all
   flag complexes.
\end{lemma}

\begin{proof}
Since $K\backslash\{v\}$ is a full subcomplex of $K$, any missing
face of $K\backslash\{v\}$ is also a missing face of~$K$. So as
$K$ is flag, any missing face has at most two elements, implying
that any missing face of $K\backslash\{v\}$ also has at most two
elements. Thus $K\backslash\{v\}$ is flag.

Let $\omega\in\MF(\starr_{K}(v))$ and $|\omega|\ge2$. We
claim that $\omega\in\MF(K)$ as well. As
$\partial\omega\subseteq\starr_{K}(v)$, we also have
$\partial\omega\subseteq K$, so if the claim does not hold then it
must be the case that $\omega\in K$. Then $v\notin\omega$, as
otherwise $\omega\in\starr_{K}(v)$. For $\tau=\omega\ast v$ we
have $\partial\tau=\partial\omega\ast v\cup\omega\in K$. As $K$ is
flag and $|\omega\ast\{v\}|>2$, we obtain $\tau=\omega\ast\{v\}\in
K$. This implies that $\omega\in\starr_{K}(v)$, a contradiction.
Hence, $\omega\in\MF(K)$ and so $|\omega|=2$ since $K$ is flag.
Thus $\starr_{K}(v)$ is flag.

Let $\omega\in\MF(\link_{K}(v))$ and $|\omega|\ge2$. By
Lemma~\ref{flaglink}, $\omega\in\MF(K\backslash\{v\})$ as well. It
has already been established that $K\backslash\{v\}$ is flag, so
we have $|\omega|=2$. Thus $\link_{K}(v)$ is also flag.
\end{proof}

Given a subset $\omega\subseteq[m]$, the full subcomplex of $K$
on $\omega$ is
\[
  K_\omega=\{\sigma\in K\mid \sigma\subseteq\omega\}.
\]
Note that $K\setminus\{v\}=K_{[m]\setminus\{v\}}$.
A key property that will be important subsequently is the
following.

\begin{lemma}
   \label{keyflag}
   Let $K$ be a flag complex on the set $[m]$ and let $v$ be a vertex
   of $K$. Then $\link_{K}(v)$ is a full subcomplex of $K\backslash\{v\}$.
\end{lemma}

\begin{proof}
Let $\omega$ be the vertex set of $\link_{K}(v)$. Suppose that
$\link_{K}(v)$ is not a full subcomplex of $K\backslash\{v\}$.
Then there is a face $\sigma\in K\backslash\{v\}$ such that
$\sigma\subseteq\omega$ and $\sigma\notin\link_{K}(v)$. By selecting
a proper face of~$\sigma$ if necessary, we may assume that
$\sigma$ is a missing face of $\link_{K}(v)$ with $|\sigma|\ge2$.
But then as $K$ is flag, Lemma~\ref{flaglink} implies that
$\sigma$ is also a missing face of $K\backslash\{v\}$. In
particular, $\sigma\notin K\backslash\{v\}$, a contradiction.~\end{proof}

\section{Homotopy theoretic preparation}
\label{sec:htpy} 

\subsection{The Cube Lemma} 
Assume that all spaces are pointed and have nondegenerate basepoints, 
implying that the inclusion of the basepoint is a cofibration. This holds, for example, 
for pointed $CW$-complexes, and hence for polyhedral products.  
One part of Mather's Cube Lemma~\cite{Ma} states that if there is a
diagram of spaces and maps
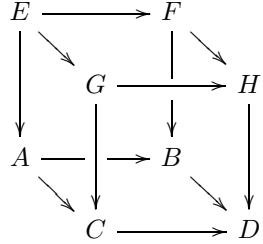
\[\spreaddiagramcolumns{-1pc}\spreaddiagramrows{-1pc}
   \diagram
   E\rrto\drto\ddto & & F\dline\drto & \\
   & G\rrto\ddto & \dto & H\ddto \\
   A\rline\drto & \rto & B\drto & \\
   & C\rrto & & D
\enddiagram\]
where the bottom face is a homotopy pushout and the four sides are
obtained by pulling back with
\(\namedright{H}{}{D}\),
then the top face is also a homotopy pushout. In what follows this will be used
to identify the homotopy type of the pushout $H$ in a certain context. However,
we need this identification to have a naturality property, which is not immediate
from the statement of the Cube Lemma. To obtain this, we prove a special
case of the Cube Lemma from first principles.

In what follows, we work with strictly commutative pushouts and pullbacks rather
than homotopy commutative ones. For a space~$Y$ let $1_{Y}$ be the identity map
on $Y$. Suppose that there is a strictly commutative diagram
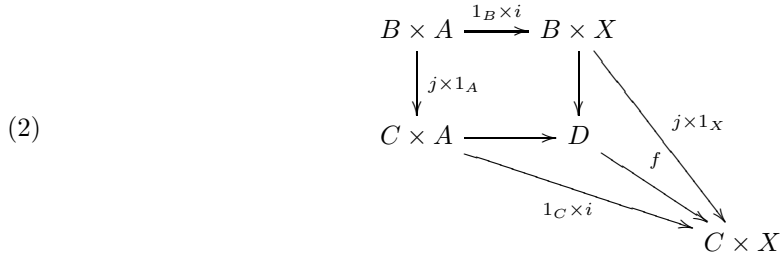
\begin{equation}
  \label{cubedata}
  \diagram
        B\times A\rto^-{1_{B}\times i}\dto^{j\times 1_{A}}
            & B\times X\dto\ddrto^{j\times 1_{X}} & \\
        C\times A\rto\drrto_{1_{C}\times i} & D\drto^-(0.4){f} & \\
           & & C\times X
  \enddiagram
\end{equation}
where the square is a pushout, and the maps $i$, $j$ and $f$ are pointed inclusions
of subspaces. We will turn the maps $f$, $1_{C}\times i$, $j\times 1_{X}$ and
$j\times i$ from the four corners of the pushout to $C\times X$ into fibrations,
up to homotopy, and examine their fibres.

There is a standard way of turning a pointed, continuous map
\(g\colon\namedright{Y}{}{Z}\)
between locally compact, Hausdorff spaces into a fibration, up to homotopy.
Let $I$ be the unit interval and let $\mbox{Map}(I,Z)$ be the space of continuous
(not necessarily pointed) maps from $I$ to $Z$. Let
\(d\colon\namedright{\mbox{Map}(I,Z)}{}{Z\times Z}\)
be defined by evaluating a map
\(\omega\colon\namedright{I}{}{Z}\)
at the two endpoints, explicitly, $d(\omega)=(\omega(0),\omega(1))$.
Define the space $\widetilde{P}_{g}$ by the pullback
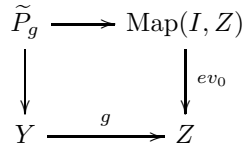
\[\diagram
       \widetilde{P}_{g}\rto\dto & \mbox{Map}(I,Z)\dto^{ev_{0}} \\
       Y\rto^-{g} & Z
  \enddiagram\]
where $ev_{0}(\omega)=\omega(0)$. As a set,
\begin{equation}
  \label{tildePpoints}
  \widetilde{P}_{g}=\{(y,\omega)\in Y\times\mbox{\rm Map}(I,Z)\mid
       \omega(0)=g(y)\}.
\end{equation}

Then, as in~\cite[p. 59]{Se} for example, there is an inclusion
\(\namedright{Y}{}{\widetilde{P}_{g}}\)
which is a homotopy equivalence and the composite
\[q\colon\nameddright{\widetilde{P}_{g}}{}{\mbox{Map}(I,Z)}{ev_{1}}{Z}\]
is a fibration, where $ev_{1}(\omega)=\omega(1)$.
Moreover, if $1$ is the basepoint of $I$ and $PZ$ is the path space of~$Z$
(with paths at time $1$ ending at the basepoint of $Z$), then the fibre of $q$
is homeomorphic to the mapping path space of $g$,
\begin{equation}
  \label{Ppoints}
  P_{g}=\{(y,\omega)\in Y\times PZ\mid\omega(0)=g(y)\},
\end{equation}
which is obtained by the pullback
\[\diagram
      P_{g}\rto\dto & PZ\dto^{ev_{0}} \\
      Y\rto^-{g} & Z.
  \enddiagram\]

Consider how these constructions behave with respect to pointed subspace inclusions. Let
\(\namedright{S}{s}{Y}\)
be the inclusion of a pointed subspace. If $Q$ is the pullback of
\(\namedright{S}{s}{Y}\)
and
\(\namedright{\widetilde{P}_{g}}{}{Y}\),
then the pullback defining~$\widetilde{P}_{g}$ implies that $Q$ is also the
pullback of $g\circ s$ and $ev_{0}$. But this pullback is the
definition of $\widetilde{P}_{g\circ s}$, so $Q=\widetilde{P}_{g\circ s}$.
Similarly for $P_{g\circ s}$, giving pullbacks
\[\diagram
       \widetilde{P}_{g\circ s}\rto\dto & \widetilde{P}_{g}\dto & &
            P_{g\circ s}\rto\dto & P_{g}\dto \\
        S\rto^-{s} & Y & & S\rto^-{s} & Y.
  \enddiagram\]
Since $P_{g}$ and $P_{g\circ s}$ are the respective fibres of $\widetilde{P}_{g}$
and $\widetilde{P}_{g\circ s}$ over $Z$, we obtain a pullback
\begin{equation}
  \label{cubepb}
  \diagram
       P_{g\circ s}\rto\dto & P_{g}\dto \\
       \widetilde{P}_{g\circ s}\rto & \widetilde{P}_{g}.
  \enddiagram
\end{equation}

Next, suppose that $Y$ is the union of pointed, closed 
subspaces $S$ and~$T$. Let
\(s\colon\namedright{S}{}{Y}\)
and
\(t\colon\namedright{T}{}{Y}\)
be the pointed subspace inclusions, and let $u$ and~$v$ be the pointed subspace inclusions
\(u\colon\namedright{S\cap T}{}{S}\)
and
\(v\colon\namedright{S\cap T}{}{T}\). 
Since $S$ and $T$ are closed subspaces of $Y$, the pushout of $u$ and~$v$ is $Y$. 
(More generally this is true if $(Y;S,T)$ is an excisive triad, but we do not need this 
level of generality - in our case each of $S$, $T$ and $Y$ will be certain polyhedral products.)  

\begin{lemma} 
   \label{tildePpo} 
   Suppose that 
   \(\namedright{Y}{g}{Z}\) 
   is a pointed subspace inclusion and that $Y=S\cup T$ where $S$ and $T$ are closed, 
   pointed subspaces of $Y$ Then there are pushouts 
   \[\diagram
         \widetilde{P}_{g\circ s\circ u}\rto\dto & \widetilde{P}_{g\circ t}\dto 
         & & P_{g\circ s\circ u}\rto\dto & P_{g\circ t}\dto \\
         \widetilde{P}_{g\circ s}\rto & \widetilde{P}_{g} & & P_{g\circ s}\rto & P_{g}. 
     \enddiagram\] 
\end{lemma} 

\begin{proof} 
By its definition, $\widetilde{P}_{g}$ is the space of paths on $Z$ that begin
in $\mbox{Im}(g)$ and end in $Z$. As $g$ is a subspace inclusion, we may
regard $\widetilde{P}_{g}$ as the space of paths on $Z$ that begin in $Y$
and end in $Z$. As $Y=S\cup T$, any such path either begins in $S$ or in $T$ -
that is - the path is either in $\widetilde{P}_{g\circ s}$ or $\widetilde{P}_{g\circ t}$.
Moreover, the intersection $\widetilde{P}_{g\circ s}\cap\widetilde{P}_{g\circ t}$
is all paths on $Z$ that begin in $S\cap T$ and end in $Z$ - that is -
the paths in $\widetilde{P}_{g\circ s\circ u}=\widetilde{P}_{g\circ t\circ v}$. 
Thus $\widetilde{P}_{g}=\widetilde{P}_{g\circ s}\cup\widetilde{P}_{g\circ t}$ and 
$\widetilde{P}_{g\circ s\circ u}=\widetilde{P}_{g\circ s}\cap\widetilde{P}_{g\circ t}$. 
Further, since $S$ and $T$ are closed subspaces of $Y$, we have 
$\widetilde{P}_{g\circ s}$ and $\widetilde{P}_{g\circ t}$ closed subspaces of 
$\widetilde{P}_{g}$. Therefore there is a pushout 
\begin{equation}
  \label{cubepo}
  \diagram
      \widetilde{P}_{g\circ s\circ u}\rto\dto & \widetilde{P}_{g\circ t}\dto \\
      \widetilde{P}_{g\circ s}\rto & \widetilde{P}_{g}.
  \enddiagram
\end{equation}
The same argument shows that $P_{g}$ is the pushout of $P_{g\circ s}$
and $P_{g\circ t}$ over $P_{g\circ s\circ u}=P_{g\circ t\circ v}$. 
\end{proof} 

Now apply this construction to the maps $f$, $1_{C}\times i$, $j\times 1_{X}$
and $j\times i$ from the four corners of the pushout in~(\ref{cubedata})
to $C\times X$.

\begin{lemma}
   \label{Pcube}
   There is a commutative cube
   \[\spreaddiagramcolumns{-1pc}\spreaddiagramrows{-1pc}
      \diagram
      P_{j\times i}\rrto\drto\ddto & & P_{j\times 1_{X}}\dline\drto & \\
      & P_{1_{C}\times i}\rrto\ddto & \dto & P_{f}\ddto \\
      \widetilde{P}_{j\times i}\rline\drto & \rto & \widetilde{P}_{j\times 1_{X}}\drto & \\
      & \widetilde{P}_{1_{C}\times i}\rrto & & \widetilde{P}_{f}
   \enddiagram\]
   where the top and bottom faces are pushouts and the four sides
   are pullbacks. Further, this cube is natural for maps of diagrams
   of the form~(\ref{cubedata}).
\end{lemma}

\begin{proof}
Since $f$, $1_{C}\times i$, $j\times 1_{X}$ and $j\times i$ are all subspace
inclusons, the four sides of the cube are pullbacks by~(\ref{cubepb}). Since $D$
is a pushout, it is the union of $C\times A$ and $B\times X$ with intersection
$B\times A$. The top and bottom faces of the cube are therefore pushouts
by~(\ref{cubepo}). The naturality statement holds since the constructions
of $\widetilde{P}_{g}$ and $P_{g}$ are natural.
\end{proof}

The top face of the cube in Lemma~\ref{Pcube} will be more precisely
identified. This requires two lemmas.

\begin{lemma}
   \label{Pdecomp-old}
   A map
   \(g\times h\colon\namedright{Y\times M}{}{Z\times N}\)
   has $P_{g\times h}=P_{g}\times P_{h}$. Further, this decomposition is natural
   for compositions
   \(s\times t\colon\namedright{Z\times N}{}{Z'\times N'}\).
\end{lemma}

\begin{proof}
First observe that $P(Z\times N)=PZ\times PN$ since any pointed path
\(\omega\colon\namedright{I}{}{Z\times N}\)
is equivalent to the product of the pointed paths
\(\omega_{1}\colon\namedright{I}{}{Z}\)
and
\(\omega_{2}\colon\namedright{I}{}{N}\)
given by projecting $\omega$ to $Z$ and $N$ respectively. Moreover, the
evaluation map
\(\namedright{P(Z\times N)}{ev_{0}}{Z\times N}\)
becomes a product of evaluation maps
\(\llnamedright{PZ\times PN}{ev_{0}\times ev_{0}}{Z\times N}\).
Thus the pullback $P_{g\times h}$ is identical to the pullback
\[\diagram
       Q\rto\dto & PZ\times PN\dto^{ev_{0}\times ev_{0}}\dto \\
       Y\times M\rto^-{s\times t} & Z\times N,
  \enddiagram\]
where
\begin{eqnarray*}
   Q & = & \{((y,m),(\omega_{1},\omega_{2})\in Y\times M\times PZ\times PN\mid
        s(y)=\omega_{1}(0), t(m)=\omega_{2}(0)\}  \\
      & = & \{(y,\omega_{1})\in Y\times PZ\mid s(y)=\omega_{1}(0)\}\times
         \{(m,\omega_{2})\in M\times PN\mid t(m)=\omega_{2}(0)\} \\
      & = & P_{s}\times P_{t}.
\end{eqnarray*}

The identification of $P_{s\times t}$ as $P_{s}\times P_{t}$ only used
the fact that $P(Z\times N)=PZ\times PN$. As the latter decomposition
is natural, therefore so is the former.
\end{proof}

\begin{lemma}
   \label{1homeo}
   There is a natural homeomorphism $P_{1_{Y}}\cong PY$.
\end{lemma}

\begin{proof}
Taking $g=1_{Y}$ in~(\ref{Ppoints}) gives
\[P_{1_{Y}}=\{(y,\omega)\in Y\times PY\mid\omega(0)=y\}.\]
Define
\(\phi\colon\namedright{PY}{}{P_{1_{Y}}}\)
by $\phi(\omega)=(\omega(0),\omega)$ and
\(\psi\colon\namedright{P_{1_{Y}}}{}{PY}\)
by $\psi(y,\omega)=\omega$. Both $\phi$ and $\psi$ are continuous,
$\psi\circ\phi=id_{PY}$ and, because for any pair $(y,\omega)\in P_{1_{Y}}$
there is the condition $y=\omega(0)$, we also have $\phi\circ\psi=1_{P_{1_{Y}}}$.
Hence $\psi$ is a homeomorphism. As both $\phi$ and $\psi$ are natural,
the homeomorphism is too.
\end{proof}

Applying Lemmas~\ref{Pdecomp-old} and~\ref{1homeo} to the top face in
Lemma~\ref{Pcube}, the space $P_{f}$ is homeomorphic to the space $Q_{f}$
defined by the pushout
\begin{equation}
  \label{Ptransform1}
  \diagram
     P_{j}\times P_{i}\rto\dto & P_{j}\times PX\dto \\
     PC\times P_{i}\rto & Q_{f}.
  \enddiagram
\end{equation}
Moreover, the naturality statements in Lemmas~\ref{Pcube} through~\ref{1homeo}
imply that~(\ref{Ptransform1}) is natural for maps of diagrams of the
form~(\ref{cubedata}).

One further modification of~(\ref{Ptransform1}) is needed. If $Y$ is a pointed
space the reduced cone on $Y$ is the space $CY=Y\wedge I$ (i.e.,
$CY=(Y\times I)/(Y\vee I)$). If $Y$ and $Z$ are pointed spaces with
basepoints $y_{0}$ and~$z_{0}$ respectively, then the reduced join is defined
by $Y\ast Z=(Y\times I\times Z)/\sim$, where $(y,0,z)=(y,0,z')$,
$(y,1,z)\sim (y',1,z)$ and $(y_{0},t,z_{0})=(y_{0},0,z_{0})$ for all
$y,y'\in Y$, $z,z'\in Z$ and $t\in I$. Observe that there is a pushout
\[\diagram
        Y\times Z\rto\dto & Y\times CZ\dto \\
        CY\times Z\rto & Y\ast Z.
  \enddiagram\]

\begin{proposition}
   \label{Pdecomp}
   Up to homotopy equivalences, the top face in Lemma~\ref{Pcube} can
   be identified with the pushout
   \[\diagram
        P_{j}\times P_{i}\rto\dto & P_{j}\times CP_{i}\dto \\
        CP_{j}\times P_{i}\rto & P_{j}\ast P_{i}.
     \enddiagram\]
   In particular, $P_{f}$ is homotopy equivalent to $P_{j}\ast P_{i}$.
   Further, this homotopy equivalence may be chosen to be natural for maps
   of diagrams of the form~(\ref{cubedata}).
\end{proposition}

\begin{proof}
In general, suppose that $Z$ is contractible. Then there is a pointed homotopy
\(\namedright{Z\times I}{}{Z}\)
which at $t=0$ is the identity map on $Z$ and at $t=1$ is the constant map
to the basepoint. The homotopy sends $Z\vee I$ to the basepoint, and so
factors through a map
\(\namedright{CZ=Z\wedge I}{}{Z}\).
That is, the contracting homotopy for $Z$ determines a specific map
\(\namedright{CZ}{}{Z}\).
If the contracting homotopy is natural for maps
\(\namedright{Z}{}{Z'}\),
then the map
\(\namedright{CZ}{}{Z}\)
is also natural. In fact, it is a natural homotopy equivalence. Refining, if
\(g\colon\namedright{Y}{}{Z}\)
is a pointed map with $Z$ being contractible, then we obtain a composite
\(\nameddright{CY}{Cg}{CZ}{}{Z}\)
with the same naturality properties.

In our case, consider~(\ref{Ptransform1}). Since $PC$ and $PX$ are contractible,
we obtain composites
\(\nameddright{P_{j}}{}{CP_{j}}{}{PC}\)
and
\(\nameddright{P_{i}}{}{CP_{i}}{}{PX}\)
in which the right hand maps are homotopy equivalences. Thus the
pushout $Q_{f}$ in~(\ref{Ptransform1}) is homotopy equivalent to the space
$P_{j}\ast P_{i}$ obtained from the pushout
\[\diagram
     P_{j}\times P_{i}\rto\dto & P_{j}\times CP_{i}\dto \\
     CP_{j}\times P_{i}\rto & P_{j}\ast P_{i}.
  \enddiagram\]
Since $P_{f}$ is homeomorphic to $Q_{f}$, we obtain $P_{f}\simeq P_{j}\ast P_{i}$.
Further, since the contracting homotopy for a path space $PZ$ can be chosen
to be natural for any map
\(\namedright{Z}{}{Z'}\),
this homotopy equivalence for $P_{f}$ is natural to the same extent
as~(\ref{Ptransform1}) is natural. That is, it is natural for maps of diagrams
of the form~(\ref{cubedata}).
\end{proof}

\subsection{Two general results on fibrations} 
Now assume that all spaces have the homotopy type of pointed $CW$-complexes. 
If $X$ is such a space then by~\cite[Corollary 3]{Mi} so is $\Omega X$. Also, any weak homotopy 
equivalence between two such spaces is a homotopy equivalence (see, for example, 
\cite[Ch.~7, \S 6, Corollary 24]{Sp}.   

\begin{lemma} 
   \label{rightinverse} 
   Suppose that 
   \(\namedddright{\Omega B}{\partial}{F}{f}{E}{p}{B}\) 
   is a homotopy fibration sequence and $p$ has a left homotopy inverse. 
   Then $\partial$ has a right homotopy inverse.  
\end{lemma} 

\begin{proof} 
Let 
\(s\colon\namedright{B}{}{E}\) 
be a map such that $s\circ p$ is homotopic to the identity map on $E$. 
Then $f\simeq s\circ p\circ f$, implying that $f$ is null homotopic since 
$p\circ f$ is. If $X$ is any pointed space then the homotopy fibration 
\(\nameddright{\Omega B}{\partial}{F}{f}{E}\) 
induces an exact sequence of pointed sets 
\(\nameddright{[X,\Omega B]}{\partial_{\ast}}{[X,F]}{f_{\ast}}{[X,B]}\) 
where $[X,Y]$ is the set of pointed homotopy classes of maps from $X$ to $Y$. 
Since $f$ is null homotopic, $f_{\ast}=0$, so $\partial_{\ast}$ is onto. 
Taking $X=F$ implies that the (homotopy class of the) identity map on $F$ 
lifts through $\partial_{\ast}$ to a map 
\(t\colon\namedright{F}{}{\Omega B}\). 
That is, $\partial\circ t$ is homotopic to the identity map on $F$.  
\end{proof}  

In general, if 
\(\nameddright{F}{f}{E}{p}{B}\) 
is a homotopy fibration where $E$ is an $H$-space and $p$ has a 
right homotopy inverse 
\(s\colon\namedright{B}{}{E}\), 
then the composite 
\[\nameddright{B\times F}{s\times f}{E\times E}{\mu}{E}\] 
is a weak homotopy equivalence, and hence a homotopy equivalence.  
We wish to give a slight variation on this in the case when $B=B_{1}\times B_{2}$ 
and each factor has a right homotopy inverse. For $i=1,2$ let $p_{i}$ be the composite 
\(p_{i}\colon\nameddright{E}{p}{B_{1}\times B_{2}}{\pi_{i}}{B_{i}}\) 
where~$\pi_{i}$ is the projection. As maps into a product are determined 
by their projection onto each factor, we have $p=(p_{1},p_{2})$.  

\begin{lemma} 
   \label{2section} 
   Let 
   \(\nameddright{F}{f}{E}{p}{B_{1}\times B_{2}}\) 
   be a homotopy fibration where $p$ is an $H$-map. Suppose that for $i=1,2$ there are maps 
   \(s_{i}\colon\namedright{B_{i}}{}{E}\) 
   such that $p_{i}\circ s_{i}$ is homotopic to the identity map on~$B_{i}$, 
   and $p_{i}\circ s_{j}$ is null homotopic for $i\neq j$. Then the composite 
   \[\llnameddright{B_{1}\times B_{2}\times F}{s_{1}\times s_{2}\times f}{E\times E\times E} 
         {\mu\circ(\mu\times 1)}{E}\] 
   is a homotopy equivalence, where $\mu$ is the multiplication on $E$. 
\end{lemma} 

\begin{proof} 
From the general result stated before the lemma, it suffices to show that 
$s_{1}\times s_{2}$ is a right homotopy inverse for $p$. Consider the diagram  
\[\diagram 
    B_{1}\times B_{2}\rto^-{s_{1}\times s_{2}}\drto_(0.4){i_{1}\times i_{2}}  
         & E\times E\rto^-{\mu}\dto^{p\times p} & E\dto^{p} \\ 
     & (B_{1}\times B_{2})\times (B_{1}\times B_{2})\rto^-{\mu'} & B_{1}\times B_{2}  
 \enddiagram\]  
where $i_{1}$ and $i_{2}$ are the inclusions into the first and second factors 
respectively and $\mu'$ is the multiplication on $B_{1}\times B_{2}$. The left triangle 
homotopy commutes since $p_{i}\circ s_{i}$ is homotopic to the identity map on 
$B_{i}$ and $p_{i}\circ s_{j}\simeq\ast$ if $i\neq j$. The right square homotopy 
commutes since $p$ is an $H$-map. Observe that the lower direction around the 
diagram is homotopic to the identity map on $B_{1}\times B_{2}$. Therefore the 
upper direction around the diagram implies that $\mu\circ(s_{1}\times s_{2})$ is 
a right homotopy inverse for $p$.  
\end{proof}

\section{Polyhedral products and the proof of Theorem~\ref{main}}

Let $K$ be a simplicial complex on the set $[m]$ and let $v$ be a
vertex of $K$. Following F\'{e}lix and Tanr\'{e}~\cite{FT}, define a new
simplicial complex $\overline{K}$ on $[m]$ by
\[\overline{K}=K\backslash\{v\}\ast\{v\}.\]
Observe that there is an inclusion of simplicial complexes
\(\namedright{K\backslash\{v\}}{}{\overline{K}}\)
given by including the join factor, so as $\starr_{K}(v)=\link_{K}(v)\ast\{v\}$,
there is a pushout map
\[\namedright{K}{}{\overline{K}}.\]
Observe also that $K\backslash\{v\}$ is the full subcomplex of
$\overline{K}$. That is, $\overline{K}\backslash\{v\}=K\backslash\{v\}$.

By~\cite{GT2}, the pushout of simplicial complexes in~(\ref{Kpoagain})
induces a pushout of polyhedral products
\begin{equation}
  \label{FTprelim}
  \diagram
         \uxa^{\link_{K}(v)}\times A_{v}\rto^-{1\times i_{v}}\dto^{j\times 1}
             & \uxa^{\link_{K}(v)}\times X_{v}\dto \\
         \uxa^{K\backslash\{v\}}\times A_{v}\rto & \uxa^{K}
   \enddiagram
\end{equation}
where $i_{v}$ is the inclusion. (Here we regard $\link_K(v)$
and $K\setminus\{v\}$ as simplicial complexes on the set $[m]\setminus\{v\}$.)
To relate this to $\uxa^{\overline{K}}$,
observe that the definition of the join of two simplicial complexes implies
that if $K=K_{1}\ast K_{2}$ then there is a homeomorphism
\[\uxa^{K}\cong\uxa^{K_{1}}\times\uxa^{K_{2}}.\]
In particular, as $\overline{K}=K\backslash\{v\}\ast\{v\}$ there is a homeomorphism
\[\uxa^{\overline{K}}\cong\uxa^{K\backslash\{v\}}\times X_{v}\]
and a strictly commutative diagram
\begin{equation}
  \label{polycubedata}
  \diagram
         \uxa^{\link_{K}(v)}\times A_{v}\rto^-{1\times i_{v}}\dto^{j\times 1}
             & \uxa^{\link_{K}(v)}\times X_{v}\dto\ddrto^-{j\times 1}  & \\
         \uxa^{K\backslash\{v\}}\times A_{v}\rto\drrto_{1\times i_{v}}
         & \uxa^{K}\drto^-(0.4){f} & \\
         & & \uxa^{K\backslash\{v\}}\times X_{v}
   \enddiagram
\end{equation}
where $f$ is the map induced by the simplicial map
\(\namedright{K}{}{\overline{K}}\)
and all maps are inclusions of subspaces.

Let $B^{K}_{v}$ be the fibre $P_{j}$ obtained by turning the map
\(\namedright{\uxa^{\link_{K}(v)}}{j}{\uxa^{K\backslash\{v\}}}\)
into a fibration and let $Y_{v}$ be the fibre $P_{i_{v}}$ obtained by turning
the inclusion
\(\namedright{A_{v}}{i_{v}}{X_{v}}\)
into a fibration.

\begin{lemma}
   \label{Ftype}
   If $F^{K}_{v}$ is the fibre $P_{f}$ obtained by turning the map
   \(\namedright{\uxa^{K}}{f}{\uxa^{K\backslash\{v\}}\times X_{v}}\)
   into a fibration, then there is a homotopy equivalence
   \[F^{K}_{v}\simeq B^{K}_{v}\ast Y_{v}.\]
   Further, this homotopy equivalence is natural for inclusions of simplicial complexes
   \(\namedright{K}{}{K'}\)
   on the set $[m]$.
\end{lemma}

\begin{proof}
Proposition~\ref{Pdecomp} immediately implies the asserted homotopy
equivalence for $F^{K}_{v}$ and states that it is natural for maps of diagrams
of the form~(\ref{polycubedata}). Now observe that any inclusion of simplicial complexes
\(\namedright{K}{}{K'}\)
on the vertex set $[m]$ induces such a map of diagrams.
\end{proof}

To take this further we need a general result about polyhedral products.

\begin{lemma}
   \label{fullsubcomplex}
   Suppose that $K_\omega$ is a full subcomplex of a simplicial complex $K$.
   Then the map of polyhedral products
   \(\namedright{\uxa^{K_\omega}}{}{\uxa^{K}}\)
   induced by the simplicial inclusion
   \(\namedright{K_\omega}{}{K}\)
   has a left inverse, that is, there is a retraction
   \(\namedright{\uxa^{K}}{}{\uxa^{K_\omega}}\). 
   Further, the construction of the left inverse is natural for simplicial inclusions  
   \(\namedright{K}{}{K'}\).  
\end{lemma}
\begin{proof}
We have
\[
  \uxa^{K}=\bigcup_{\sigma\in\mathcal K}
  \Bigl(\prod_{i\in \sigma}X_i\times\prod_{i\in[m]\setminus\sigma}A_i\Bigl),\quad
  \uxa^{K_\omega}=\bigcup_{\sigma\in\mathcal K,\,\sigma\subseteq\omega}
  \Bigl(\prod_{i\in \sigma}X_i\times\prod_{i\in \omega\setminus\sigma}A_i\Bigl).
\]
Since each $A_i$ is a pointed space, there is a canonical
inclusion \(\namedright{\uxa^{K_\omega}}{}{\uxa^{K}}\).
Furthermore, for each $\sigma\in K$ there is a projection
\[
  r_\sigma\colon \prod_{i\in \sigma}X_i\times
  \prod_{i\in[m]\setminus \sigma}A_i\longrightarrow
  \prod_{i\in \sigma\cap\omega}X_i\times\prod_{i\in \omega\setminus \sigma}A_i.
\]
Since $K_\omega$ is a full subcomplex, the image of $r_\sigma$
belongs to $\uxa^{K_\omega}$. The projections $r_\sigma$ patch
together to give a retraction $r=\bigcup_{\sigma\in
K}r_\sigma\colon\namedright{\uxa^{K}}{}{\uxa^{K_\omega}}$. 
The naturality assertion follows from the naturality of inclusions and projections.  
\end{proof}

\begin{proposition}
   \label{uxasplit}
   Let $K$ be a simplicial complex on the index set $[m]$ and let $v$ be a vertex
   of $K$. Then there is a homotopy equivalence
   \[\Omega\uxa^{K}\simeq\Omega X_{v}\times\Omega\uxa^{K\backslash\{v\}}
           \times\Omega(B^{K}_{v}\ast Y_{v})\]
   which is natural for inclusions of simplicial complexes
   \(\namedright{K}{}{K'}\)
   on the set $[m]$. 
\end{proposition} 

\begin{proof} 
Consider the homotopy fibration
\begin{equation}\label{1fib}
  \nameddright{F^{K}_{v}}{}{\uxa^{K}}{f}{\uxa^{K\backslash\{v\}}
  \times X_{v}} 
\end{equation}
from Lemma~\ref{Ftype}. Observe that $K\backslash\{v\}$ and $\{v\}$ are the full
subcomplexes of $K$ on the sets $[m]-\{v\}$ and $\{v\}$
respectively. So by Lemma~\ref{fullsubcomplex}, the maps 
\(s_{1}\colon\namedright{\uxa^{K\backslash\{v\}}}{}{\uxa^{K}}\) 
and 
\(s_{2}\colon\namedright{X_{v}}{}{\uxa^{K}}\) 
have left inverses 
\(\namedright{\uxa^{K}}{f_{1}}{\uxa^{K\backslash\{v\}}}\) 
and 
\(\namedright{\uxa^{K}}{f_{2}}{X_{v}=\uxa^{\{v\}}}\) 
respectively. Since the vertex sets for $K\backslash\{v\}$ and $\{v\}$ 
are disjoint, the left inverses have the property that $f_{1}\circ s_{2}$ 
and $f_{2}\circ s_{1}$ are trivial. Lemma~\ref{2section} cannot be 
applied immediately since $f$ is usually not an $H$-map, but after looping 
the homotopy fibration~(\ref{1fib}) it can be applied, and this gives the asserted 
homotopy equivalence. 

The naturality property follows from the naturality properties of the simplicial map 
\(\namedright{K}{}{K\backslash\{v\}\ast\{v\}}\), 
the polyhedral product and Lemma~\ref{fullsubcomplex}, together with the 
fact that 
\(\namedright{\Omega\uxa^{K}}{}{\Omega\uxa^{K'}}\) 
is an $H$-map. 
\end{proof}

One more preliminary result is needed before the proof of Theorem~\ref{main}.
Let $K$ be a simplicial complex on the vertex set $[m]$, let $K^{f}$ be the flagification 
of $K$, and let $L$ be the simplicial complex consisting of the vertices of $K$. 
Let $M$ be either $L$ or $K$. If $v$ is a vertex of $K$ then the simplicial map 
\(\namedright{M}{}{K^{f}}\)  
implies that there is commutative diagram of simplicial complexes
\[\diagram
        \link_{M}(v)\rto\dto & M\backslash\{v\}\dto \\
        \link_{K^{f}}(v)\rto & K^{f}\backslash\{v\}.
  \enddiagram\]
Taking polyhedral products and then taking homotopy fibres gives a homotopy
fibration diagram
\begin{equation}
  \label{Bdgrm}
  \diagram
        \Omega\uxa^{M\backslash\{v\}}\rto\dto & B^{M}_{v}\rto\dto^{b_{v}}
            & \uxa^{\link_{M}(v)}\rto\dto & \uxa^{M\backslash\{v\}}\dto \\
        \Omega\uxa^{K^{f}\backslash\{v\}}\rto & B^{K^{f}}_{v}\rto
            & \uxa^{\link_{K^{f}}(v)}\rto & \uxa^{K^{f}\backslash\{v\}}
  \enddiagram
\end{equation}
for some induced map of fibres $b_{v}$.

\begin{lemma}
   \label{Bsplitting}
   Let $M$ be either $L$ or $K$. Suppose that in~\eqref{Bdgrm} the map
   \(\namedright{\Omega\uxa^{M\backslash\{v\}}}{}{\Omega\uxa^{K^{f}\backslash\{v\}}}\)
   has a right homotopy inverse. Then $b_{v}$ has a right homotopy inverse
   \(s_{v}\colon\namedright{B^{K^{f}}_{v}}{}{B^{M}_{v}}\).
   Moreover, $s_v$ can be chosen so that it factors through the map
   \(\namedright{\Omega\uxa^{M\backslash\{v\}}}{}{B^{M}_{v}}\).
\end{lemma}

\begin{proof} 
Consider the homotopy fibration along the bottom row of~(\ref{Bdgrm}).  
Since $K^{f}$ is flag, by Lemma~\ref{keyflag}, $\link_{K^{f}}(v)$ is a
full subcomplex of $K^{f}\backslash\{v\}$. Thus
$\uxa^{\link_{K^{f}}(v)}$ is a retract of $\uxa^{K^{f}\backslash\{v\}}$. 
Therefore, by Lemma~\ref{rightinverse}, the map
\(\namedright{\Omega\uxa^{K^{f}\backslash\{v\}}}{}{B^{K^{f}}_{v}}\)
has a right homotopy inverse
\(t\colon\namedright{B^{K^{f}}_{v}}{}{\Omega\uxa^{K^{f}\backslash\{v\}}}\).
By hypothesis, the map
\(\namedright{\Omega\uxa^{M\backslash\{v\}}}{}{\Omega\uxa^{K^{f}\backslash\{v\}}}\)
has a right homotopy inverse
\(s\colon\namedright{\Omega\uxa^{K^{f}\backslash\{v\}}}{}
           {\Omega\uxa^{M\backslash\{v\}}}\).
Thus there is a homotopy commutative diagram
\[\diagram
      B^{K^{f}}_{v}\rto^-{t}
           & \Omega\uxa^{K^{f}\backslash\{v\}}\rto^-{s}\drdouble
           & \Omega\uxa^{M\backslash\{v\}}\rto\dto & B^{M}_{v}\dto^{b_{v}} \\
      & & \Omega\uxa^{K^{f}\backslash\{v\}}\rto & B^{K^{f}}_{v}.
  \enddiagram\]
As the lower direction around the diagram is homotopic to the identity
map on $B^{K^{f}}_{v}$, the upper direction around the diagram
implies that $b_{v}$ has a right homotopy inverse.
\end{proof}

\begin{proof}[Proof of Theorem~\ref{main}]
Let $K$ be a simplicial complex on the vertex set $[m]$, let
$K^{f}$ be its flagification, and let $L$ be $m$ disjoint points.
Then there is a sequence of inclusions of simplicial complexes
\(\nameddright{L}{}{K}{}{K^{f}}\). Taking polyhedral products with
respect to $\uxa$ gives a sequence of maps
\(h\colon\nameddright{\uxa^{L}}{g}{\uxa^{K}}{f}{\uxa^{K^{f}}}\)
We will show that $\Omega h$ has a right homotopy inverse,
implying that the map
\(\Omega f\colon\namedright{\Omega\uxa^{K}}{}{\Omega\uxa^{K^{f}}}\) also has a
right homotopy inverse. This would prove both parts of the
statement of the theorem.

The proof is by induction on the number of vertices. If $m=1$, then
$L$, $K$ and $K^{f}$ all equal the single vertex $\{1\}$, implying
that $h$ is the identity map, and so $\Omega h$ has a right homotopy inverse.
Assume that the statement of the theorem holds for all simplicial complexes
with strictly less than~$m$ vertices. The decomposition and naturality statements
in Proposition~\ref{uxasplit} imply that there is a homotopy commutative diagram
of homotopy equivalences
\begin{equation}
  \label{mainequivdgrm}
  \diagram
      (\Omega\uxa^{L\backslash\{v\}}\times\Omega X_{v})
          \times\Omega(B^{L}_{v}\ast Y_{v})
          \rto^-{\simeq}\dto^{(\Omega a\times 1)\times\Omega(b_{v}\ast 1)}
        & \Omega\uxa^{L}\dto \\
     (\Omega\uxa^{K^{f}\backslash\{v\}}\times\Omega X_{v})
          \times\Omega(B^{K^{f}}_{v}\ast Y_{v})\rto^-{\simeq}
        & \Omega\uxa^{K^{f}}.
  \enddiagram
\end{equation}
Observe that $K^{f}\backslash\{v\}$ has $m-1$ vertices and
\(\namedright{L\backslash\{v\}}{}{K^{f}\backslash\{v\}}\)
is the inclusion of these vertices. Since $L$ and $K^{f}$ are flag complexes,
by Lemma~\ref{slrflag} so are $L\backslash\{v\}$ and $K^{f}\backslash\{v\}$.
Therefore, by inductive hypothesis, the map $\Omega a$ has a right homotopy inverse
\(s\colon\namedright{\Omega\uxa^{K^{f}\backslash\{v\}}}{}
       {\Omega\uxa^{L\backslash\{v\}}}\).
As $L$ and $K^{f}$ are flag complexes and $\Omega a$ has a right
homotopy inverse, by  Lemma~\ref{Bsplitting} the map $b_{v}$ also has a
right homotopy inverse
\(t\colon\namedright{B^{K^{f}}_{v}}{}{B_{v}^{L}}\).
Therefore $t'=\Omega(t\ast 1)$ is a right homotopy inverse for $\Omega(b_{v}\ast 1)$.
Putting $s$ and $t'$ together we obtain a map
\[\llnamedright{\Omega\uxa^{K^{f}\backslash\{v\}}\times\Omega X_{v}\times
         \Omega(B^{K^{f}}_{v}\ast Y_{v})}{s\times 1\times t'}
         {\Omega\uxa^{L\backslash\{v\}}\times\Omega X_{v}\times
         \Omega(B^{L}_{v}\ast Y_{v})}\]
which is a right homotopy inverse of $(\Omega a\times 1)\times\Omega(b_{v}\ast 1)$.
The homotopy equivalences in~(\ref{mainequivdgrm}) therefore imply that the map
\(h\colon\namedright{\Omega\uxa^{L}}{}{\Omega\uxa^{K^{f}}}\)
has a right homotopy inverse. This completes the induction.
\end{proof}

\section{Refinements}
\label{sec:refinements}

This section gives two refinements describing the homotopy type
of the space $B^{K}_{v}$ under certain conditions.
First consider the homotopy fibration diagram~(\ref{Bdgrm}) in the 
case when $M=K$. Define
the space $D_{v}^{K}$ and the map~$d_{v}$ by the homotopy fibration
\begin{equation}
  \label{DKfib}
  \nameddright{D_{v}^{K}}{d_{v}}{B_{v}^{K}}{b_{v}}{B_{v}^{K^{f}}}.
\end{equation}

\begin{lemma}
   \label{Bdecomp}
   Given the hypotheses of Lemma~\ref{Bsplitting}, there is a homotopy equivalence
   $B_{v}^{K}\simeq B_{v}^{K^{f}}\times D_{v}^{K}$.
\end{lemma}

\begin{proof}
By Lemma~\ref{Bsplitting}, $b_{v}$ has a right homotopy inverse
\(\namedright{B_{v}^{K^{f}}}{s_{v}}{B_{v}^{K}}\).
As $B_{v}^{K}$ need not be an $H$-space this does not immediately
imply that it is homotopy equivalent to $B_{v}^{K^{f}}\times D_{v}^{K}$.
However, Lemma~\ref{Bsplitting} also says that $s_{v}$ can be chosen
to factor through the homotopy fibration connecting map
\(\namedright{\Omega\uxa^{K\backslash\{v\}}}{}{B_{v}^{K}}\).
That is, $s_{v}$ can be chosen to be a composite
\(\nameddright{B_{v}^{K^{f}}}{s'_{v}}{\Omega\uxa^{K\backslash\{v\}}}{}{B_{v}^{K}}\)
for some map $s'_{v}$. For any homotopy fibration sequence
\(\namedddright{\Omega B}{\delta}{F}{}{E}{}{B}\)
the connecting map~$\delta$ satisfies a homotopy action
\(\theta\colon\namedright{\Omega B\times F}{}{F}\)
which restricts to the identity map on $F$ and~$\delta$ on~$\Omega B$.
In our case, we obtain a composite
\[\psi\colon\llnameddright{B_{v}^{K^{f}}\times D_{v}^{K}}{s'_{v}\times d_{v}}
       {\Omega\uxa^{K\backslash\{v\}}\times B_{v}^{K}}{\theta}{B_{v}^{K}}.\]
Observe that the restriction of $\psi$ to $B_{v}^{K^{f}}$
is $s_{v}$ and the restriction to $D_{v}^{K}$ is $d_{v}$. Thus $\psi$ is a
trivialization of the homotopy fibration~(\ref{DKfib}), implying that it is a
homotopy equivalence.
\end{proof}

Second, suppose that $K$ is a flag complex. By Lemma~\ref{keyflag},
$\link_{K}(v)$ is a full subcomplex of $K\backslash\{v\}$. So by
Lemma~\ref{fullsubcomplex}, the inclusion
\(\namedright{\uxa^{\link_{K}(v)}}{}{\uxa^{K\backslash\{v\}}}\) has
a left inverse. Define $C^{K}_{v}$ by the homotopy fibration
\begin{equation}
  \label{Bdeloop}
  \nameddright{C^{K}_{v}}{}{\uxa^{K\backslash\{v\}}}{}{\uxa^{\link_{K}(v)}}.
\end{equation}
From the retraction of $\uxa^{\link_{K}(v)}$ off $\uxa^{K\backslash\{v\}}$
and the definitions of $B^{K}_{v}$ and $C^{K}_{v}$ we obtain a homotopy pullback diagram
  \[\diagram
       B^{K}_{v}\rdouble\dto & B^{K}_{v}\dto & \\
       \ast\rto\dto & \uxa^{\link_{K}(v)}\rdouble\dto & \uxa^{\link_{K}(v)}\ddouble \\
       C^{K}_{v}\rto & \uxa^{K\backslash\{v\}}\rto & \uxa^{\link_{K}(v)}.
  \enddiagram\]
Thus $B^{K}_{v}\simeq\Omega C^{K}_{v}$.

\begin{lemma}
   \label{uxaflagsplit}
   Let $K$ be a flag complex on the vertex set $[m]$ and let $v$ be a vertex of $K$.
   Then there are homotopy equivalences
   \begin{align*}
       & \Omega\uxa^{K}\simeq\Omega X_{v}\times\Omega\uxa^{K\backslash\{v\}}
            \times\Omega(\Omega C^{K}_{v}\ast Y_{v}) \\
      & \Omega\uxa^{K\backslash\{v\}}\simeq\Omega\uxa^{\link_{K}(v)}\times\Omega C^{K}_{v}.
   \end{align*}
\end{lemma}

\begin{proof}
The first homotopy equivalence follows immediately from
Proposition~\ref{uxasplit}, while the second is an immediate
consequence of the homotopy fibration~(\ref{Bdeloop}) and the
retraction of $\uxa^{\link_{K}(v)}$ off $\uxa^{K\backslash\{v\}}$.
\end{proof}

\section{Co-$H$-space properties}
\label{sec:coH}

In this section we consider polyhedral products of the form $\cyy^{K}$
and identify the class of flag complexes $K$ for which $\cyy^{K}$ is a
co-$H$-space. As a corollary, we obtain conditions that allow for a
delooping of the statement of Theorem~\ref{main}. This begins with an
abstract lemma.

\begin{lemma}
   \label{delooplemma}
   Let $A$ and $B$ be pointed spaces with the homotopy types of
   $CW$-complexes. Suppose that there is a pointed map
   \(f\colon\namedright{A}{}{B}\)
   and $B$ is a co-$H$-space. If $\Omega f$ has a right homotopy inverse
   then $f$ has a right homotopy inverse.
\end{lemma}

\begin{proof}
Since $B$ is a co-$H$-space, by~\cite{G2} there is a map
\(s\colon\namedright{B}{}{\Sigma\Omega B}\)
which is a right homotopy inverse to the canonical evaluation map
\(ev\colon\namedright{\Sigma\Omega B}{}{B}\).
Let
\(t\colon\namedright{\Omega B}{}{\Omega A}\)
be a right homotopy inverse of $\Omega f$. Consider the diagram
\[\diagram
         B\rto^-{s} & \Sigma\Omega B\dto^-{\Sigma t}\drdouble & \\
         & \Sigma\Omega A\rto^-{\Sigma\Omega f}\dto^{ev}
              & \Sigma\Omega B\dto^{ev} \\
         & A\rto^-{f} & B.
  \enddiagram\]
The upper triangle homotopy commutes since $t$ is a right homotopy
inverse of $\Omega f$. The lower square homotopy commutes by the
naturality of the evaluation map. The upper direction around the diagram
is homotopic to $ev\circ s$, which is the identity map on $B$. The lower
direction around the diagram therefore implies that $ev\circ\Sigma t\circ s$
is a right homotopy inverse of $f$.
\end{proof}

\begin{proposition}
   \label{delooping}
   Let $K$ be a simplicial complex on the vertex set $[m]$, let $K^{f}$
   be the flagification of $K$, and let $Y_{1},\ldots,Y_{m}$ be pointed $CW$-complexes.
   If $\cyy^{K^{f}}$ is homotopy equivalent to a co-$H$-space then the map
   \(f\colon\namedright{\cyy^{K}}{}{\cyy^{K^{f}}}\)
   induced by the simplicial inclusion 
   \(\namedright{K}{}{K^{f}}\)
   has a right homotopy inverse.
\end{proposition}

\begin{proof}
Taking $\uxa=\cyy$, by Theorem~\ref{main},
$\Omega f\colon\Omega\cyy^K\longrightarrow\Omega\cyy^{K^f}$
has a right homotopy inverse. Since $\cyy^{K^{f}}$ is a
co-$H$-space, Lemma~\ref{delooplemma} implies that $f$
has a right homotopy inverse.
\end{proof}

\begin{remark}
Note that in Proposition~\ref{delooping} we do not need to assume that
$Y_1,\ldots,Y_m$ are path-connected. Since we asssume that every singleton
of $[m]$ is a vertex ($K$ is on the vertex set~$[m]$), $\cyy^K$ is path-connected
even if $Y$ is not.
\end{remark}

Next we obtain a characterisation of those flag complexes $K$ for
which $\cyy^K$ is a co-$H$-space. In terms of notation, when all
pairs in the sequence $\{(X_{i},A_{i})\}_{i=1}^{m}$ are the same,
$(X_i,A_i)=(X,A)$, we use the notation $(X,A)^K$ for~$\uxa^K$.
Special cases are the Davis-Januskiewicz space
$DJ(K)=(\mathbb{C}P^{\infty},\ast)^{K}$ and the moment-angle
complex $\zk=(D^{2},S^{1})^{K}$.

A graph $\Gamma$ is called \emph{chordal} if each of its
cycles with $\ge 4$ vertices has a chord (an edge joining two
vertices that are not adjacent in the cycle). Equivalently, a
chordal graph is a graph with no induced cycles of length more
than three. By the result of Fulkerson and Gross~\cite{fu-gr65} a
graph is chordal if and only if its vertices can be ordered in
such a way that, for each vertex~$i$, the lesser neighbours of~$i$
form a clique. Such an order of vertices is called a \emph{perfect
elimination ordering}.

By~\cite{GPTW}, $\mathcal{Z}_{K^{f}}=(D^2,S^1)^{K^f}$ is homotopy
equivalent to a wedge of spheres if and only if the $1$-skeleton
of $K^{f}$ is a chordal graph. In particular, if the $1$-skeleton
of $K^{f}$ is a chordal graph then $\mathcal{Z}_{K^{f}}$ is a
co-$H$-space. This result is readily extended to general
polyhedral products of the form $\cyyk$, where $CY$ denotes the
cone over~$Y$. Let $X^{\vee k}$ be the $k$-fold wedge of~$X$.

\begin{theorem}\label{theo:chordec}
Assume that $K$ is a flag complex on the vertex set~$[m]$ and
$\widetilde H^*(Y_i;\mathbb Z)\ne0$ for $1\le i\le m$. The
following conditions are equivalent
\begin{itemize}
\item[(a)] the $1$-skeleton $K^1$ is a chordal graph;

\item[(b)] $\cyyk$ is a co-$H$-space.
\end{itemize}
Furthermore, if $K^1$ is chordal, there is a homotopy equivalence
\begin{equation}\label{chordec}
  \cyyk\simeq\bigvee_{k=2}^m\quad\bigvee_{1\le i_1<\cdots<i_k\le m}
  \bigl(\Sigma Y_{i_1}\wedge\cdots\wedge Y_{i_k}
  \bigr)^{\vee\, c(i_1,\ldots,i_k)},
\end{equation}
where $c(i_1,\ldots,i_k)=\mathop\mathrm{rank}\widetilde
H^0(K_{\{i_1,\ldots,i_k\}})$ is one less than the number of connected
components of the full subcomplex~$K_{\{i_1,\ldots,i_k\}}$.
\end{theorem}

\begin{proof}
The argument is similar to~\cite[Theorem~4.6]{GPTW}
or~\cite[Theorem~4.3]{pa-ve16}, but this time we keep track of the
wedge summands. Assume that $K^1$ is chordal. Choose a perfect
elimination ordering of vertices, and for each vertex
$i=1,\ldots,m$ denote by $\sigma_i$ the face of $K$ corresponding
to the clique of $K^1$ consisting of $i$ and its lesser
neighbours. All maximal faces of $K$ are among
$\sigma_1,\ldots,\sigma_m$, so we have
$\bigcup_{i=1}^m\sigma_i=K$. Furthermore, for each $k=1,\ldots,m$
the perfect elimination ordering on $K$ induces such an ordering
on the full subcomplex $K_{\{1,\ldots,k-1\}}$, so we have
$\bigcup_{i=1}^{k-1}\sigma_i=K_{\{1,\ldots,k-1\}}$. In particular,
the simplicial complex $\bigcup_{i=1}^{k-1}\sigma_i$ is flag as a
full subcomplex in a flag complex. The intersection
$\sigma_k\cap\bigcup_{i=1}^{k-1}\sigma_i$ is a clique
$\sigma_k\setminus\{k\}$, so it is a face of
$\bigcup_{i=1}^{k-1}\sigma_i$. Therefore, $K$ is obtained by
iteratively attaching $\sigma_k$ to $\bigcup_{i=1}^{k-1}\sigma_i$
along the common face~$\sigma_k\setminus\{k\}$.

We use induction on $m$ to prove the
decomposition~\eqref{chordec}. When $m=1$, both sides
of~\eqref{chordec} are trivial. Now assume that~\eqref{chordec}
holds for $K$ with $<m$ vertices. The pushout
square~\eqref{Kpoagain} for $v=\{m\}$ becomes
\[
  \diagram
         \sigma_m\setminus\{m\}\rto\dto & \sigma_m\dto \\
         K\setminus\{m\}\rto & K.
\enddiagram
\]
According to our convention, $\sigma_m\setminus\{m\}$ and
$K\setminus\{m\}$ are regarded as simplicial complexes
on~$[m]\setminus\{m\}=[m-1]$, while $\sigma_m$ is regarded as a
complex on~$[m]$. The corresponding pushout square~\eqref{FTprelim}
of the polyhedral products becomes
\begin{equation}\label{pushoutm}
  \diagram
         \cyy^{\sigma_m\setminus\{m\}}\times Y_m\rto\dto^{j\times 1}
             & \cyy^{\sigma_m}\dto \\
         \cyy^{K\setminus\{m\}}\times Y_m\rto & \cyy^{K}
   \enddiagram
\end{equation}
As $\sigma_m\setminus\{m\}$ is a face of $K\setminus\{m\}$ and
$\sigma_m$ is a face of $K$, we have
\[
  \cyy^{\sigma_m\setminus\{m\}}=\prod_{i\in\sigma_m\setminus\{m\}}CY_i
  \times \prod_{i\notin\sigma_m}Y_i,
  \quad
  \cyy^{\sigma_m}=\prod_{i\in\sigma_m}CY_i
  \times \prod_{i\notin\sigma_m}Y_i.
\]
Since each $\{i\}$ is a vertex of $K$, the inclusion
$\prod_{i\in\omega}Y_i\to\cyy^K$ is null-homotopic for any
subset $\omega\subseteq[m]$, and the same holds with $K$ replaced
by~$K\setminus\{m\}$. Hence, the map
$j\times1$ in~\eqref{pushoutm} decomposes into the composition
$i_2\circ\pi_2$ of the projection onto the second factor and the
inclusion. It follows that the pushout square~\eqref{pushoutm}
decomposes as
\[
  \diagram
         \prod_{i\notin\sigma_m}Y_i\times Y_m\rto^-{\pi_1}\dto^{\pi_2}
          & \prod_{i\notin\sigma_m}Y_i \dto \\
         Y_m \rto^-\epsilon \dto^{i_2}
          & \bigl(\prod_{i\notin\sigma_m}Y_i\bigr)\ast Y_m \dto\\
         \cyy^{K\setminus\{m\}}\times Y_m\rto & \cyy^{K}
  \enddiagram
\]
where the map $\epsilon$ is null-homotopic. From the bottom
pushout square we obtain
\[
  \cyy^K\simeq\Bigl(\cyy^{K\setminus\{m\}}\rtimes Y_m\Bigr)\vee
  \Bigl(\bigr(\prod_{i\notin\sigma_m}Y_i\bigl)\ast Y_m\Bigr),
\]
where $X\rtimes Y=X\times Y/(*\times Y)$ is the right half-smash
product, which is homotopy equivalent to $X\vee(X\wedge Y)$ when
$X$ is a suspension. By the inductive hypothesis,
$\cyy^{K\setminus\{m\}}$ is a suspension, so we can rewrite the
identity above as
\[
  \cyy^K\simeq\cyy^{K\setminus\{m\}}\vee
  \bigl(\cyy^{K\setminus\{m\}}\wedge Y_m\bigr)\vee
  \Bigl(\mathop{\bigvee_{1\le i_1<\cdots<i_k\le m-1}}
  \limits_{\{i_j,m\}\notin K}
  \Sigma Y_{i_1}\wedge\cdots\wedge Y_{i_k}\wedge Y_m\Bigr),
\]
Now a simple counting argument together with the inductive
hypothesis gives~\eqref{chordec}. This also proves the
implication (a)$\Rightarrow$(b).

To prove the implication (b)$\Rightarrow$(a), assume that
$K^1$ is not chordal. Choose an induced chordless cycle $K_\omega$
with $|\omega|\ge4$ (i.\,e. a full subcomplex isomorphic to the
boundary of an $|\omega|$-gon). Then there is a nontrivial product
in the cohomology ring $H^*(\cyy^{K_\omega};\mathbb Z)$. (When
$(\underline{CY},\underline{Y})=(D^1,S^0)$, the polyhedral product
$(D^1,S^0)^{K_\omega}$ is an orientable surface of positive
genus~\cite[Example~6.40]{BP1}; the general case then follows
from~\cite[Theorem~1.9]{BBCG2}). By Lemma~\ref{fullsubcomplex},
the same nontrivial product appears in $H^*(\cyy^{K};\mathbb Z)$.
Thus, $\cyy^K$ is not a co-$H$-space.
\end{proof}

\begin{remark}
Theorem~\ref{theo:chordec} implies that the wedge decomposition of
$\Sigma\cyyk$ of~\cite{BBCG} desuspends when $K$ is flag and $K^1$
is chordal; this also follows from the results of Iriye and
Kishimoto~\cite[Theorem~1.2, Proposition~3.2]{IK}. Other classes
of simplicial complexes $K$ with this property are described
in~\cite{IK} and~\cite{GT3}. The novelty of Theorem~\ref{theo:chordec} 
compared to~\cite{IK} is the description of the wedge decomposition 
of $\cyyk$ in terms of the degree zero cohomology of full subcomplexes of~$K$, 
which does not follow readily from desuspending the decomposition 
in~\cite{BBCG}.

When $K$ is not flag, the implication (b)$\Rightarrow$(a) of
Theorem~\ref{theo:chordec} still holds, but (a)$\Rightarrow$(b)
fails. Indeed one can take $K$ to be the boundary of a
\emph{cyclic polytope}~\cite[Example~1.1.17]{BP2} of dimension
$n\ge4$ with $m>n+1$ vertices. Then $K^1$ is a complete graph on
$m$ vertices, so it is chordal. On the other hand, $\mathcal
Z_K=(D^2,S^1)^K$ is an $(m+n)$-manifold with nontrivial cohomology
product, so it cannot be a co-$H$-space.
\end{remark}

Finally, we give conditions that allow for a delooping of the maps in Theorem~\ref{main}.

\begin{corollary}
   \label{chordal}
   Let $K$ be a simplicial complex on the vertex set $[m]$ whose $1$-skeleton
   is a chordal graph. If $K^{f}$ is the flagification of $K$ then the map
   \(f\colon\namedright{\cyy^{K}}{}{\cyy^{K^{f}}}\)
   has a right homotopy inverse.~$\qqed$
\end{corollary}

\begin{proof}
As $K$ and $K^f$ have the same $1$-skeleton,
Theorem~\ref{theo:chordec} implies that $\cyy^{K^{f}}$ is a
co-$H$-space (and even a suspension). The result follows from
Proposition~\ref{delooping}.
\end{proof}

\begin{corollary}
   \label{chordal1}
   Let $K$ be a flag simplicial complex on the vertex set $[m]$, and let $L$ be
   the simplicial complex given by $m$ disjoint points.  The map
   \(h\colon\namedright{\cyy^{L}}{}{\cyy^{K}}\)
   has a right homotopy inverse if and only if the $1$-skeleton of $K$ is a chordal graph.
\end{corollary}

\begin{proof}
Assume that $K^1$ is a chordal graph. As $K$ is flag, Theorem~\ref{main}
implies that $\Omega h$ has a right homotopy inverse, and
Theorem~\ref{theo:chordec} implies that $\cyy^{K}$ is a
co-$H$-space. Then $h$ has a right homotopy inverse by
Lemma~\ref{delooplemma}.

Now assume that $h$ has a right homotopy inverse. Then $\cyy^{K}$ is a
co-$H$-space, being a retract of the co-$H$-space $\cyy^{L}$.
Theorem~\ref{theo:chordec} implies that $K^1$ is a chordal graph.
\end{proof}

\begin{remark}
Given
\(\nameddright{\cyy^{L}}{g}{\cyy^{K}}{f}{\cyy^{K^{f}}}\),
Theorem~\ref{main} states that each of the two maps $\Omega f$
and $\Omega h=\Omega f\circ\Omega g$ has a right homotopy inverse.
Corollary~\ref{chordal} gives a sufficient condition for a delooping of
the first map, and Corollary~\ref{chordal1} gives a necessary and
sufficient condition for a delooping of the second map.
In both cases the condition is that $K^1$ is a chordal graph. However,
this condition is obviously not necessary for a delooping of $\Omega f$. 
Indeed, $f$ has a right inverse for any flag~$K$, not only for those with chordal~$K^1$,
because in this case $K^{f}=K$ and $f$ is the identity map.
\end{remark}


\section{Whitehead products}
\label{sec:whitehead}

In this section we describe two ways of relating the results of Theorem~\ref{main} and
Theorem~\ref{theo:chordec} to the classical iterated Whitehead products. First, we consider polyhedral products of the form $\ux^{K}$ with flag $K$ whose $1$-skeleton is a chordal graph, and obtain a generalisation (Proposition~\ref{Whitehead}) of Porter's description of the homotopy fiber of the inclusion of an $m$-fold wedge into a product in terms of Whitehead brackets.
Second, we consider the loop space $\Omega(\underline{S},\underline{\ast})^{K}$ on a polyhedral product of spheres for an arbitrary flag complex~$K$, and obtain a generalisation (Proposition~\ref{HMS}) of the Hilton--Milnor Theorem.

First, specialize to the case when each pair $(X_{i},A_{i})$ is of the form $(X_{i},\ast)$
and write $\ux$ for $\uxa$. By~\cite{GT1}, for example, there is a homotopy
fibration
\[\nameddright{\clxx^{K}}{\gamma_{K}}{\ux^{K}}{}{\prod_{i=1}^{m} X_{i}}\]
for any simplicial complex $K$. This is natural for simplicial inclusions, so if $K$
is a flag complex on the vertex set $[m]$ and
\(\namedright{L}{}{K}\)
is the inclusion of the vertex set then there is a homotopy fibration diagram
\begin{equation}
  \label{LKwh}
  \diagram
       \clxx^{L}\rto^-{\gamma_{L}}\dto^{h'} & \ux^{L}\rto\dto^{h}
           & \prod_{i=1}^{m} X_{i}\ddouble \\
       \clxx^{K}\rto^-{\gamma_{K}} & \ux^{K}\rto & \prod_{i=1}^{m} X_{i}
  \enddiagram
\end{equation}
where both $h$ and $h'$ are induced maps of polyhedral products. By
Theorem~\ref{main}, $\Omega h'$ has a right homotopy inverse. Further,
if $K^{1}$ is a chordal graph then Proposition~\ref{delooping} and Theorem~\ref{theo:chordec} imply that $h'$ has a right homotopy inverse.

Observe that as $L$ is $m$ disjoint points we have $\ux^{L}=X_{1}\vee\cdots\vee X_{m}$,
implying that $\clxx^{L}$ is the homotopy fibre of the inclusion of the wedge
into the product. Porter~\cite{P} identified the homotopy type of this fibre, from
which we obtain a homotopy equivalence
\begin{equation}
   \label{Porterdecomp}
   \clxx^{L}\simeq\bigvee_{k=2}^{m}\ \ \bigvee_{1\leq i_{1}<\cdots<i_{k}\leq m}
   (\Sigma\Omega X_{i_{1}}\wedge\cdots\wedge\Omega X_{i_{k}})^{\vee (k-1)}.
\end{equation}
Notice that $L^{1}$ is a chordal graph and the decomposition in~(\ref{Porterdecomp})
exactly matches that of $\clxx^{L}$ in~(\ref{chordec}). Moreover,
by~\cite[Theorem 6.2]{T}, Porter's homotopy type identification can be chosen so
that the composite
\[\varphi_{L}\colon\nameddright{\bigvee_{k=2}^{m}\ \
   \bigvee_{1\leq i_{1}<\cdots<i_{k}\leq m}
   (\Sigma\Omega X_{i_{1}}\wedge\cdots\wedge\Omega X_{i_{k}})^{\vee (k-1)}}
    {\simeq}{\clxx^{L}}{\gamma_{L}}{\ux^{L}}\]
is a wedge sum of iterated Whitehead products of the maps
\[ev_{i}\colon\namedright{\Sigma\Omega X_{i}}{ev}{X_{i}}\hookrightarrow
      X_{1}\vee\cdots\vee X_{m}=\ux^{L}.\]
Returning to~(\ref{LKwh}), the naturality of the Whitehead product implies
that $h\circ\varphi_{L}$ is a wedge sum of Whitehead products mapping
into $\ux^{K}$. The right homotopy inverse for $h'$
when $K^{1}$ is a chordal graph therefore implies the following.

\begin{proposition}
   \label{Whitehead}
   Let $K$ be a flag complex such that $K^{1}$ is a chordal graph. Then the map
   \(\namedright{\clxx^{K}}{\gamma_{K}}{\ux^{K}}\)
   factors through a wedge sum of Whitehead products.~$\qqed$
\end{proposition}

In the case when $\ux^{K}= (\cpinf,\ast)^{K}=DJ(K)$ and $\clxx^{K}\simeq(D^{2},S^{1})^{K}=\zk$ the result above follows from~\cite[Theorem~4.3]{GPTW}, where the Whitehead products were explicitly specified as iterated brackets of the canonical generators.

Theorem~\ref{main} also leads to a generalization of the Hilton-Milnor Theorem.
In this case we specialize to pairs $(\Sigma X_{i},\ast)$, giving
$(\underline{\Sigma X},\underline{\ast})^{L}=\Sigma X_{1}\vee\cdots\vee\Sigma X_{m}$.
The Hilton-Milnor Theorem states that there is a homotopy equivalence
\begin{equation}
  \label{HM}
  \Omega(\Sigma X_{1}\vee\cdots\vee\Sigma X_{m})\simeq
     \prod_{\alpha\in L\langle V\rangle}
     \Omega(\Sigma X_{1}^{\wedge\alpha_{1}}\wedge\cdots\wedge X_{m}^{\wedge\alpha_{m}})
\end{equation}
where: $V$ is a free $\mathbb{Z}$-module on $m$ elements $x_{1},\ldots,x_{m}$;
$L\langle V\rangle$ is the free Lie algebra on $V$; $\alpha$ runs over
a $\mathbb{Z}$-module basis of $L\langle V\rangle$;
and $\alpha_{i}$ is the number of occurances of $x_{i}$ in the bracket $\alpha$.
Here, if $\alpha_{i}=0$ we interpret $X_{i}$ as being omitted from the smash
product rather than as being trivial. For example,
$X_{1}^{\wedge 2}\wedge X_{2}^{0}=X_{1}^{\wedge 2}$. The Hilton-Milnor
Theorem also describes the maps from the factors on the right side of~(\ref{HM})
into $\Omega(\Sigma X_{1}\vee\cdots\vee\Sigma X_{m})$. If the length
of $\alpha$ is $1$ then the relevant factor is $\Omega\Sigma X_{i}$ for some $i$
and the map
\(\namedright{\Omega\Sigma X_{i}}{}{\Omega(\Sigma X_{1}\vee\cdots\vee\Sigma X_{m})}\)
is the loops on the inclusion into the wedge. If the length of $\alpha$ is larger
than $1$ then the map
\(\namedright{\Omega(\Sigma X_{1}^{\wedge\alpha_{1}}\wedge\cdots\wedge
      X_{m}^{\wedge\alpha_{m}})}{}{\Omega(\Sigma X_{1}\vee\cdots\vee\Sigma X_{m})}\)
is the loops on the Whitehead product corresponding to the bracket $\alpha$.

By Theorem~\ref{main}, if $K$ is a flag complex on the vertex set $[m]$ then
the map
\(\namedright{\Omega(\underline{\Sigma X},\underline{\ast})^{L}}{h}
        {\Omega(\underline{\Sigma X},\underline{\ast})^{K}}\)
has a right homotopy inverse. In particular,
$\Omega(\underline{\Sigma X},\underline{\ast})^{K}$
is a retract of the product on the right side of~(\ref{HM}). It is probably the
case that the retraction consists of selecting an appropriate subproduct,
but this is not immediately clear. That is, simply knowing that
$\Omega h$ has a right homotopy inverse leaves open the possibility that
some of the factors
$\Omega(\Sigma X_{1}^{\wedge\alpha_{1}}\wedge\cdots\wedge X_{m}^{\wedge\alpha_{m}})$
split as $A\times B$ where $A$ retracts off
$\Omega(\underline{\Sigma X},\underline{\ast})^{K}$
while $B$ does not. However, if we specialize a bit more then this possibility
is essentially eliminated.

Suppose that each $X_{i}$ is a connected sphere $S^{n_{i}-1}$ and
write $(\underline{S},\underline{\ast})$ for
$(\underline{\Sigma X},\underline{\ast})$. Since each $X_{i}$ is a sphere, the space
$\Sigma X_{1}^{\wedge\alpha_{1}}\wedge\cdots\wedge X_{m}^{\wedge\alpha_{m}}$
is homotopy equivalent to a sphere, so the right side of~(\ref{HM}) becomes
a product of looped spheres. The space $\Omega S^{n}$ is indecomposable
unless $n\in\{2,4,8\}$. In the latter case, we have a homotopy equivalence $\Omega H\times E\colon\namedright{\Omega S^{2n-1}\times S^{n-1}} {\simeq} {\Omega S^n}$, which is a product of the looped Hopf map $H$ and the suspension map~$E$. The retraction of
$\Omega(\underline{S},\underline{\ast})^{K}$ off $\Omega(\underline{S},\underline{\ast})^{L}$
implies the following.

\begin{proposition}
   \label{HMS}
   Let $K$ be a flag complex. Then
   \[\Omega(\underline{S},\underline{\ast})^{K}\simeq\left(\prod_{i=1}^{m}\Omega S^{n_{i}}\right)\times M\]
   where $M$ is homotopy equivalent to a product of spheres and loops on spheres. Further,
   \begin{itemize}
   \item[(a)] a factor $\Omega S^{n}$ of $M$ with $n\notin\{3,7,15\}$ maps to
   $\Omega(\underline{S},\underline{\ast})^{K}$ by a looped Whitehead product
   \(\namedright{\Omega S^{n}}{\Omega w}{\Omega(\underline{S},\underline{\ast})^{K}}\);

   \item[(b)] a factor $\Omega S^{2n-1}$ of $M$ with $n\in\{2,4,8\}$ maps to
   $\Omega(\underline{S},\underline{\ast})^{K}$ by a looped Whitehead product
   \(\namedright{\Omega S^{2n-1}}{\Omega w}{\Omega(\underline{S},\underline{\ast})^{K}}\)
  or by a composite \(\nameddright{\Omega S^{2n-1}}{\Omega H}{\Omega S^{n}}{\Omega w}
        {\Omega(\underline{S},\underline{\ast})^{K}}\),
   where $H$ is the Hopf map;

   \item[(c)]
   a factor $S^{n-1}$ of $M$ has $n\in\{2,4,8\}$ and maps
   to $\Omega(\underline{S},\underline{\ast})^{K}$ by a composite
   \(\nameddright{S^{n-1}}{E}{\Omega S^{n}}{\Omega w}
        {\Omega(\underline{S},\underline{\ast})^{K}}\),
   where $E$ is the suspension map and $w$ is a Whitehead product.~$\qqed$
   \end{itemize}
\end{proposition}

Refining a bit, by~\cite{GT1} the homotopy fibration
\(\nameddright{(\underline{C\Omega S},\underline{\Omega S})^{K}}{\gamma_{K}}
      {(\underline{S},\underline{\ast})^{K}}{}{\prod_{i=1}^{m} S^{n_{i}}}\)
splits after looping to give a homotopy equivalence
\[\Omega(\underline{S},\underline{\ast})^{K}\simeq(\prod_{i=1}^{m}\Omega S^{n_{i}})
        \times\Omega(\underline{C\Omega S},\underline{\Omega S})^{K}.\]
Therefore, Proposition~\ref{HMS} implies that if $K$ is a flag complex then
$\Omega(\underline{C\Omega S},\underline{\Omega S})^{K}$ is homotopy
equivalent to a product of spheres and loops on spheres, and under this
homotopy equivalence $\Omega\gamma_{K}$ becomes a product of maps of the from
$\Omega w$, $\Omega w\circ\Omega H$ or $\Omega w\circ E$.

This has implications for moment-angle complexes and Davis-Januszkiewicz
spaces. Recall that $DJ(K)\simeq (\cpinf,\ast)^{K}$ and
$\zk\simeq (D^{2},S^{1})^{K}$. There is a homotopy fibration
\[\nameddright{\zk}{\psi_{K}}{DJ(K)}{}{\prod_{i=1}^{m}\cpinf}\]
which splits after looping to give a homotopy equivalence
\[\Omega DJ(K)\simeq (\prod_{i=1}^{m} S^{1})\times\Omega\zk.\]

The inclusion
\(\namedright{S^{2}}{}{\cpinf}\)
induces maps of pairs
\(\namedright{(S^{2},\ast)}{}{(\cpinf,\ast)}\) and
\(\nameddright{(C\Omega S^{2},\Omega
S^{2})}{}{(C\Omega\cpinf,\Omega\cpinf)}
      {\simeq}{(D^{2},S^{1})}\).
These then induce a commutative diagram of polyhedral products
\begin{equation}
   \label{FGdgrm}
   \diagram
        (C\Omega S^{2},\Omega S^{2})^{K}\rto^-{\gamma_{K}}\dto^{G}
            & (S^{2},\ast)^{K}\dto^{F} \\
        \zk\rto^-{\psi_{K}} & DJ(K).
   \enddiagram
\end{equation}

Observe that the suspension map
\(\namedright{S^{1}}{E}{\Omega S^{2}}\)
induces a map of pairs
\(\namedright{(CS^{1}, S^{1})}{}{(C\Omega S^{2},\Omega S^{2})}\)
with the property that the composite
\(\nameddright{(CS^{1}, S^{1})}{}{(C\Omega S^{2},\Omega S^{2})}{}
     {(D^{2},S^{1})}\)
is a homotopy equivalence. This implies that the map $G$
in~(\ref{FGdgrm}) has a right homotopy inverse. If $K$ is a flag complex
then Proposition~\ref{HMS} says that
$\Omega(C\Omega S^{2},\Omega S^{2})^{K}$
is homotopy equivalent to a product of spheres and loops on spheres,
and the factors map to $\Omega(S^{2},\ast)^{K}$
by maps of the form $\Omega w$, $\Omega w\circ\Omega H$ or
$\Omega w\circ E$. Thus from the map $G$ in~(\ref{FGdgrm})
having a right homotopy inverse, and $F$ being natural with
respect to Whitehead products, we obtain the following.

\begin{corollary}
   \label{zktype}
   Let $K$ be a flag complex. Then $\Omega\zk$ is homotopy equivalent
   to a product of spheres and loops on spheres, and under this equivalence
   the map
   \(\namedright{\Omega\zk}{\Omega\psi_{K}}{\Omega DJ(K)}\)
   becomes a product of maps of the form $\Omega w$, $\Omega w\circ\Omega H$
   or $\Omega w\circ E$ where $w$ is a Whitehead product.~$\qqed$
\end{corollary}

Notice that $\zk$ itself is often not a product or a wedge of spheres.
For example, if $K$ is the boundary of an $n$-gon for $n\geq 5$
then $K$ is flag and $\zk$ is diffeomorphic to a connected sum of
products of two spheres. Nevertheless, $\Omega\zk$ is homotopy
equivalent to a product of spheres and loops on spheres.

\section{Homotopy theoretic consequences}
\label{sec:apps}

We restrict attention to Davis-Januszkiewicz spaces
$DJ(K)=(\mathbb{C}P^{\infty},\ast)^{K}$ and moment-angle complexes
$\zk=(D^{2},S^{1})^{K}$. Let \(\namedright{S^{2}}{}{\cpinf}\) be
the inclusion of $S^{2}\cong\mathbb{C}P^{1}$ into $\cpinf$. Then
there is an induced map of polyhedral products
\[i_{K}\colon\namedright{(S^{2},\ast)^{K}}{}{(\cpinf,\ast)^{K}}.\]
Building on the fact that the map $G$ in~(\ref{FGdgrm}) has a right homotopy inverse,
in~\cite{GT3} the following was proved.

\begin{lemma}
   \label{S2cpinf}
   The map $\Omega i_{K}$ has a right homotopy inverse.~$\qqed$
\end{lemma}

\begin{lemma}
   \label{conntriv}
   Let $K$ be a flag complex. Suppose that there is a map
   \(h\colon\namedright{(\cpinf,\ast)^{K}}{}{Y}\)
   where~$Y$ is $2$-connected. Then $\Omega h$ is null homotopic. Consequently,
   $h$ induces the zero map on homotopy groups.
\end{lemma}

\begin{proof}
Let $L$ be the simplicial complex on $m$ disjoint points. The simplicial map
\(\namedright{L}{}{K}\)
induces a map of polyhedral products
\(g\colon\namedright{(S^{2},\ast)^{L}}{}{(S^{2},\ast)^{K}}\).
Consider the composite
\[\namedddright{(S^{2},\ast)^{L}}{g}{(S^{2},\ast)^{K}}{i_{K}}{(\cpinf,\ast)^{K}}{h}{Y}.\]
Observe that by the definition of the polyhedral product,
$(S^{2},\ast)^{L}\simeq\bigvee_{i=1}^{m} S^{2}$. Since $Y$ is
$2$-connected, the composite $h\circ i_{K}\circ g$ is therefore
null homotopic. Since $K$ is a flag complex, by
Theorem~\ref{main}, $\Omega g$ has a right homotopy inverse.
Therefore $\Omega h\circ\Omega i_{K}$ is null homotopic. By
Lemma~\ref{S2cpinf}, $\Omega i_{K}$ also has a right homotopy
inverse. Therefore $\Omega h$ is null homotopic.
\end{proof}

For example, let $C$ be the homotopy cofibre of the composite
\[\psi\colon\nameddright{\bigvee_{i=1}^{m} S^{2}}{}{\bigvee_{i=1}^{m}\mathbb{C}P^{\infty}}{}{DJ(K)}\]
where the left map is the wedge of inclusions of the bottom cells and the right map
is the map of polyhedral products induced by including the vertices into $K$. The
description of $\cohlgy{DJ(K);\mathbb{Z}}$ (see, for example~\cite{BP1}) implies that
$C$ is $3$-connected. Therefore Lemma~\ref{conntriv} implies that if $K$ is a flag
complex then the quotient map
\[f\colon\namedright{DJ(K)}{}{C=DJ(K)/(\bigvee_{i=1}^{m} S^{2})}\]
induces the trivial map on homotopy groups.

Lemma~\ref{conntriv} says that if $K$ is a flag complex then the bottom
$2$-spheres in $DJ(K)$ have a great impact on its  homotopy theory. The next
lemma says this much more dramatically in the case of~$\zk$ when~$K^{1}$ is a
chordal graph.

\begin{lemma}
   \label{zkfactor}
   Let $K$ be a flag complex such that $K^{1}$ is a chordal graph. Then
   there is a homotopy commutative diagram
   \[\diagram
         & \bigvee_{i=1}^{m} S^{2}\dto^{\psi} \\
         \zk\rto\urto^{\lambda} & DJ(K)
     \enddiagram\]
   for some map $\lambda$.
\end{lemma}

\begin{proof}
As usual, let $L$ be the vertex set of $K$. Consider the diagram
\[\diagram
        (C\Omega S^{2},\Omega S^{2})^{L}\rto^-{\gamma_{L}}\dto^{h}
            & (S^{2},\ast)^{L}\dto \\
        (C\Omega S^{2},\Omega S^{2})^{K}\rto^-{\gamma_{K}}\dto^{G}
            & (S^{2},\ast)^{K}\dto^{F} \\
        \zk\rto^-{\psi_{K}} & DJ(K).
  \enddiagram\]
The upper square is induced by the simplicial inclusion of $L$ into $K$.
The lower square homotopy commutes by~(\ref{FGdgrm}). Notice that the
right column is equal to $\psi$.  As mentioned in the previous section, the
map $G$ has a right homotopy inverse. Since
$K$ is a flag complex and $K^{1}$ is a chordal graph, by Corollary~\ref{chordal1},
the map $h$ has a right homotopy inverse. Thus $G\circ h$ has a right
homotopy inverse and the lemma follows.
\end{proof}

\bibliographystyle{amsalpha}

\end{document}